\documentclass[a4paper,10pt]{amsart}
\usepackage{etex}

\usepackage[latin1]{inputenc}
\usepackage{amsmath, amsthm, amssymb, comment}
\usepackage{amscd}
\usepackage[dvips]{graphicx}
\usepackage[all]{xy}
\usepackage{enumitem}
\usepackage{hyperref}
\hypersetup{
    colorlinks,
    linkcolor={red!50!black},
    citecolor={blue!50!black},
    urlcolor={blue!80!black}
}

\usepackage{caption}
\usepackage{subcaption}
\usepackage[rightcaption]{sidecap}

\usepackage[usenames,dvipsnames]{pstricks}
\usepackage{epsfig}
\usepackage{pst-grad} 
\usepackage{pst-plot} 
\usepackage{pstricks-add}
\usepackage{pst-solides3d}

\usepackage[margin=3cm]{geometry}

\usepackage{color}

\newtheorem{theorem}{Theorem}[section]

\newtheorem{lemma}[theorem]{Lemma}

\newtheorem*{theorem*}{Theorem}
\newtheorem*{proposition*}{Proposition}
\newtheorem*{definition*}{Definition}
\newtheorem*{lemma*}{Lemma}
\newtheorem*{claim*}{Claim}
\newtheorem*{corollary*}{Corollary}

\newtheorem{thmintro}{Theorem}

\theoremstyle{definition}

\theoremstyle{remark}
\newtheorem{rem}[theorem]{Remark}
\newtheorem*{rem*}{Remark}

\newcommand{\wt}[1]{\widetilde{#1}}

\newcommand\R{\mathbb R}

\newcommand\Z{\mathbb Z}

\newcommand\cZ{\mathcal Z}
\newcommand\cW{\mathcal W}
\newcommand\cP{\mathcal P}

\newcommand\eps{\varepsilon}

\newcommand{\vol}{\mathrm{Vol}}

\newcommand\flot{ \phi^{t} }
\newcommand\hflot{ \tilde{ \phi}^t }

\newcommand{\mcg}{\mathrm{MCG}(M)}
\renewcommand\phi\varphi

\title[Centralizers of partially hyperbolic diffeomorphisms]{Centralizers of partially hyperbolic diffeomorphisms in dimension 3}
\author{Thomas Barthelm\'e}
\thanks{The first author was partially supported by the NSERC (Funding reference number RGPIN-2017-04592).}
\address{Queen's University, Kingston, Ontario}
\email{thomas.barthelme@queensu.ca}
\urladdr{sites.google.com/site/thomasbarthelme}

\author{Andrey Gogolev}
 \thanks{The second author was partially supported by NSF DMS-1823150.}
\address{Ohio State University, Columbus, Ohio}
\email{gogolyev.1@osu.edu}

\begin{document}

\begin{abstract}
In this note we describe centralizers of volume preserving partially hyperbolic diffeomorphisms which are homotopic to identity on Seifert fibered and hyperbolic 3-manifolds. Our proof follows the strategy of Damjanovic, Wilkinson and Xu~\cite{DWX} who recently classified the centralizer for perturbations of  time-$1$ maps of geodesic flows in negative curvature. We strongly rely on recent classification results in dimension 3 established in~\cite{BFFP}.
\end{abstract}

\maketitle

In \cite{DWX}, Damjanovic, Wilkinson and Xu investigate centralizers of certain partially hyperbolic diffeomorphisms and prove the following beautiful rigidity result: The centralizer of a perturbation of a time-$1$ map of an Anosov geodesic flow is either virtually $\Z$ or it is virtually $\R$. In the latter case the partially hyperbolic diffeomorphism is the time-$1$ map of a smooth Anosov flow.

The proof in~\cite{DWX} works equally well in any dimension. Here we point out that, if one considers only $3$-manifolds, then some lemmas can be strengthened to obtain the rigidity result for a much broader class of partially hyperbolic diffeomorphisms.

For any diffeomorphism $f\colon M \to M$, we denote the centralizer of $f$ by
\[
\cZ (f):=\lbrace g\in \mathrm{Diff}(M) \mid g\circ f = f\circ g\rbrace, 
\]
where $\mathrm{Diff}(M)$ is the space of $C^1$-diffeomorphisms of $M$.

We say that $f\colon M \to M$ is a \emph{discretized Anosov flow} if $f$ is a partially hyperbolic diffeomorphism such that there exists a (topological) Anosov flow $\flot \colon M \to M$ and a function $h\colon M \to \R^+$ such that $f(x) = \phi^{h(x)}(x)$ for all $x\in M$\footnote{This is the same definition as in \cite{BFFP}, see Appendix G of \cite{BFFP} for more details. Note that a discretized Anosov flow is a much broader class than what is called a discretized flow in \cite{DWX}, which is just a time-$1$ map of an Anosov flow.}.

In this note, the partially hyperbolic diffeomorphism $f$ is always assumed to be a $C^{\infty}$ diffeomorphism.

\begin{thmintro}\label{thm_main}
 Let $f\colon M \to M$ be a volume-preserving partially hyperbolic diffeomorphism on a $3$-manifold. If $f$ is a discretized Anosov flow and $\pi_1(M)$ is not virtually solvable then either $\cZ(f)$ is virtually $\left\{ f^n \mid n\in \Z\right\}$ or $f$ embeds into a smooth Anosov flow.
\end{thmintro}

Using the main results of \cite{BFFP}, we then deduce the following results.

\begin{thmintro}\label{thm_seifert}
 Let $f\colon M \to M$ be a volume-preserving partially hyperbolic diffeomorphism on a Seifert $3$-manifold which is homotopic to the identity. Then either $\cZ(f)$ is virtually $\left\{ f^n \mid n\in \Z\right\}$ or $\cZ(f)$ is virtually $\R$ and a power of $f$ embeds into an Anosov flow.
\end{thmintro}

\begin{thmintro}\label{thm_hyperbolic}
 Let $f\colon M \to M$ be a volume-preserving dynamically coherent partially hyperbolic diffeomorphism on a hyperbolic $3$-manifold. Then either $\cZ(f)$ is virtually $\left\{ f^n \mid n\in \Z\right\}$ or $\cZ(f)$ is virtually $\R$ and a power of $f$ embeds into an Anosov flow.
\end{thmintro}

%
%
%
\begin{rem}
Note that Theorem \ref{thm_seifert} is a generalization of the $3$-dimensional case of Theorem~3 of~\cite{DWX}. (One has to take a power of $f$ to obtain the embedding into an Anosov flow only in the case when $M$ is a $k$-cover of the unit tangent bundle of a hyperbolic surface or an orbifold, see Remark 7.4 in \cite{BFFP}).
\end{rem}

\begin{rem}
 The reason we exclude virtually solvable $\pi_1(M)$ in Theorem \ref{thm_main} is that, in this case, $f$ would be a discretized Anosov flow of a suspension of an Anosov diffeomorphism. Thus $f$ could fail to be accessible and the main motor of the proof, which is a dichotomy result by Avila, Viana and Wilkinson \cite{AVW1,AVW2}, does not work. If one asks for $f$ to be accessible, then Theorem \ref{thm_main} will apply even on manifold with virtually solvable fundamental group.
 
 In particular, any dynamically coherent, accessible, volume-preserving partially hyperbolic diffeomorphism $f$ on a manifold with virtually solvable fundamental group has centralizer virtually $\Z$ or virtually $\R$ in which case a power of $f$ embeds into an Anosov flow. (The proof follows as in section \ref{sec_proof_of_B_C}, but using the classification results of Hammerlindl and Potrie (see \cite{HPSurvey}) instead of \cite{BFFP}).
\end{rem}

As this note heavily relies on the arguments of \cite{DWX} to obtain Theorem \ref{thm_main}, we did not try to make it self-contained and refer to \cite{DWX} whenever an argument does not need substantial change.

\section{Proof of Theorem \ref{thm_main}}

Overall the proof follows the scheme of the proof of Theorem 3 of \cite{DWX}. The difference is in the following lemmas which are more general (when considering the 3-dimensional case) from their counterparts in \cite{DWX}.

For $f\colon M \to M$ a dynamically coherent partially hyperbolic diffeomorphism, we denote by $\cW^{s}$, $\cW^u$, $\cW^{cs}$, $\cW^{cu}$, and $\cW^c$ the stable, unstable, center stable, center unstable and center foliations of $f$, respectively. 
Recall that the foliations $\cW^{s}$ and $\cW^u$ are unique, but, in general, the others are not. Thankfully, for discretized Anosov flow, they are unique.

\begin{lemma}\label{lem_uniquefoliation}
 Let $f\colon M \to M$ be a discretized Anosov flow. Then there exists a unique pair of center stable $\cW^{cs}$ and center unstable $\cW^{cu}$ foliations that are preserved by $f$. Hence $\cW^c$ is also unique.
\end{lemma}

\begin{proof}
 Since $f$ is a discretized Anosov flow, it admits a pair of center stable and center unstable foliations such that a good lift $\wt f$ of $f$ to the universal cover $\wt M$ fixes each leaf of the lifted foliations (see \cite[Proposition G.1]{BFFP}). Thus, by \cite[Lemma 12.6]{BFFP}, these foliations are unique.
\end{proof}

As a direct consequence of Lemma \ref{lem_uniquefoliation}, we obtain that, if $g\in \cZ(f)$, then $g$ preserves each of the foliations $\cW^{\ast}$, $\ast = c,s,u,cs,cu$.

Following \cite{DWX}, denote by $\cZ^c(f)$ the subgroup of $\cZ(f)$ consisting of elements which fix each leaf of the center foliation of $f$.

Let $\mcg= \pi_0\left(\mathrm{Diff}(M)\right)$ be the mapping class group of $M$. Denote by $\cZ_0(f)$ the kernel of the homomorphism $\cZ(f) \to \mcg$. Note that $\cZ^c(f)$ is a subgroup of $\cZ_0(f)$. Indeed, on the universal cover the leaf space is $\R^2$ and each center leaf is a line and, hence, $g\in\cZ^c(f)$ can be homotoped to the identity along the center leaves.
\begin{lemma}\label{lem_kernel_fix_center}
Let $f\colon M\to M$ be a discretized Anosov flow, and suppose that the corresponding Anosov flow $\flot$ is transitive.
 Then, the group $\cZ^c(f)$ has finite index in the kernel $\cZ_0(f)$.
%
\end{lemma}

\begin{proof}
 Suppose that $g\in \cZ_0(f)$.
 Since $f$ is a discretized Anosov flow, its center foliation $\cW^c$ is the orbit foliation of a topological Anosov flow $\flot$ (cf.~\cite[Proposition G.1]{BFFP}).
By the preceding lemma $g$ preserves the foliation $\cW^c$. Thus the map $g$ is a self orbit equivalence of the transitive Anosov flow $\flot$ which is homotopic to the identity. Therefore Theorem 1.1 of \cite{BG} applies to $g$.

Then, either $g\in \cZ^c(f)$ or (see case 4 of \cite[Theorem 1.1]{BG}) $\phi^t$ is $\R$-covered and there exists a map $\eta\colon M \to M$ and an integer $i$ such that $g\circ \eta^i$ fixes every leaf of $\cW^c$.

Since $g$ is at least $C^1$, if $i\neq 0$, then $g$ defines a non-trivial $C^1$ action on the weak-stable leaf space of the Anosov flow $\flot$, and thus, by \cite[Proposition 6.6]{BarHDR}, $\flot$ is a finite cover of the geodesic flow on a (orientable) hyperbolic surface or orbifold $\Sigma$. That is, we are in case 4b of Theorem 1.1 of \cite{BG}, so the map $\eta$ can be chosen to be the lift of the rotation by $2\pi$ along the fiber of $T^1\Sigma$, call it $r$. In particular, we have $r^k=\mathrm{Id}$.

Hence, we obtained a homomorphism $\cZ_0(f)/\cZ^c(f)\ni[g]\mapsto i\in \Z/ k\Z$ which is injective. Thus $\cZ_0(f)/\cZ^c(f)$ is finite.
%
%
%
%
%
%
\end{proof}

\begin{lemma}\label{lem_power_fix_closed_center}
 Let $f\colon M\to M$ be a discretized Anosov flow. Then for any $g\in \cZ(f)$ and any closed center leaf $\cW^c(x)$, there exists $k\geq 1$ such that 
 \[
g^k\left( \cW^c(x) \right) = \cW^c(x).
 \]
\end{lemma}
\begin{proof}
 This is essentially the same proof as Lemma 23 in \cite{DWX}, but we rewrite it since we state it in a different setting.
 
 Let $\flot\colon M\to M$ be the topological Anosov flow and $h\colon M \to \R^+$ be the continuous function such that $f(x)= \phi^{h(x)}(x)$. We fix a metric on $M$ such that the orbits of $\flot$ have unit speed.
 
 Let $g\in \cZ(f)$.
  Let $\hflot$, $\wt f$ and $\wt g$ be lifts of $\flot$, $f$ and $g$ to the universal cover $\wt M$. We choose $\hflot$ and $\wt f$ to be lifts which fix each leaf of the lifted center foliation $\wt \cW^c$ (= the flow foliation of $\hflot$). If $g$ reverses the orientation of the orbits of $\flot$, then we replace $g$ by $g^2$. Thus we can assume that $\wt g$ preserves the ordering of points on any orbit of $\hflot$.
  
  Recall that all orbits of $\hflot$ are lines. Hence a closed center leaf $\cW^c(x)$ lifts to an orbit segment
  $[x, \wt\phi^T(x)]$, $T>0$ (where we write $x$ for both the point $x\in M$ and a lift of it to the universal cover $\wt M$). The orbit of $x$ under $\wt f$ is an increasing sequence of points. Hence, there exists a unique $N\ge 0$ such that $\wt\phi^T(x)$ belongs to the orbit segment $(\wt f^Nx, \wt f^{N+1}x]$. Then, for any $m\ge 1$, the points $\wt g^m\wt\phi^Tx$ belongs to the orbit segment $(\wt g^m(\wt f^Nx), \wt g^m(\wt f^{N+1}x)]=(\wt f^N(\wt g^mx), \wt f^{N+1}(\wt g^mx)]$. 
  
  The center leaf $\cW^c(g^mx)$ lifts to the orbit segment $[\wt g^m x, \wt g^m(\wt\phi^Tx)]$. By the above discussion we have $[\wt g^m x, \wt g^m(\wt\phi^Tx)]\subset [\wt g^m x, \wt f^{N+1}(\wt g^mx)]$. Hence the length of $\cW^c(g^mx)$ is bounded by $C=(N+1)\max(h)$. Note that this bound is uniform in $m$.
 
 Since there are only finitely many closed center leaves of length less than $C$, it follows that every closed center leaf is $g$-periodic. 

\end{proof}

\begin{lemma}\label{lem_quotient_is_finite}
 Let $f\colon M\to M$ be a discretized Anosov flow, and suppose that the Anosov flow $\flot$ is transitive. 
 Then $\cZ(f)/\cZ^c(f)$ is finite.
\end{lemma}

\begin{proof}
By Lemma~\ref{lem_kernel_fix_center}, since $\cZ^c(f)$ has finite index in $\cZ_0(f)$, it is sufficient to show that $\cZ(f)/\cZ_0(f)$ is finite, which we now proceed to do.

 Let $g\in \cZ(f)$.
 By Lemma \ref{lem_power_fix_closed_center}, every closed center leaf in $\cW^c$ is periodic under $g$. Now recall that each closed center leaf is a periodic orbit of the transitive Anosov flow $\flot$. By \cite{Adachi}, the (conjugacy classes of) closed orbits of $\flot$ generate the fundamental group of $M$. Thus we can choose a generating set of closed orbits and choose $n$ large enough so that $g^n$ fixes each closed center leaf in the generating set of conjugacy classes of $\pi_1(M)$.
 
 This implies that the element $[g^n_\ast]\in \mathrm{Out}(\pi_1(M))$ is the identity of the outer automorphism group of $\pi_1(M)$.
 
 Thus $g^n$, seen as an element of $\mcg$, is in the kernel of the homomorphism $\mcg \to \mathrm{Out}(\pi_1(M))$.

%
%
%

A standard obstruction theory argument shows that, when $M$ is aspherical (which is the case here, because $M$ is $3$-dimensional and supports an Anosov flow), the map  $\mcg \to \mathrm{Out}(\pi_1(M))$ is injective. Thus $g^n$ is the identity in $\mcg$. Hence, we conclude that $\cZ(f)/\cZ_0(f)$ is a torsion subgroup of $\mcg$.
 

 Now, since $M$ is an irreducible $3$-manifold, $\mcg$ is virtually torsion free (see section 5 of \cite{MCG_3_manifolds})\footnote{Note that McCullough \cite{McCullough_Haken} proved that $\mcg$ is virtually torsion free for Haken manifolds and it follows from Mostow Rigidity Theorem for hyperbolic manifolds, which are the only two cases we need, since, as $M$ supports an Anosov flow, it is either Haken or hyperbolic.}.
 Thus, $\cZ(f)/\cZ_0(f)$  must be finite, since it is a torsion subgroup of $\mcg$.
%
%
\end{proof}

We now have all needed lemmas and we can copy verbatim the proof of Theorem~5 of \cite{DWX} and obtain the following result that will allow us to deduce Theorem \ref{thm_main}.

\begin{theorem}\label{thm_dichotomy}
 Let $f$ be a discretized Anosov flow on a 3-manifold $M$ such that $\pi_1(M)$ is not virtually solvable. Suppose that $f$ preserves a volume $\vol$ on $M$.
 Then either $\vol$ has Lebesgue disintegration along $\cW^c$ or $f$ has virtually trivial centralizer in $\mathrm{Diff}(M)$.
\end{theorem}

\begin{proof}
As $\pi_1(M)$ is not virtually solvable, by \cite[Theorem C]{FP}, $f$ is accessible. Because $f$ is volume preserving, it is, thus, transitive (\cite{Brin_transitive}). Hence there exists a center leaf which is dense in $M$, which implies that the Anosov flow $\flot$ is also transitive. So all of the lemmas we proved above apply.
 
We have that $\cZ(f)$ is virtually $\cZ^c(f)$.
 Moreover, $f$ is ergodic (because it is accessible, so it is ergodic by~\cite{RHRHU, BurnsW}) and all the elements of $\cZ(f)$ are volume preserving (see \cite[Lemma 11]{DWX}).

From the the proof of Theorem H of \cite{AVW2} (see section 10.3 of \cite{AVW2}) we have the following lemma.

 \begin{lemma}\label{lem_sing_desintegration}
  If $\vol$ has singular disintegration along the leaves of $\cW^c$, then there exists $k\geq 1$ and a full measure set $S\subset M$ that intersects every center leaf in exactly $k$ orbits of $f$.
 \end{lemma}
 
 This lemma replaces Lemma 51 of \cite{DWX}, and one can now copy verbatim the proof of Theorem 5 in \cite{DWX} (replacing $T^1X$ with $M$) to obtain Theorem \ref{thm_dichotomy}.
\end{proof}

\begin{proof}[Proof of Theorem \ref{thm_main}]
 If $\vol$ has singular disintegration along the leaves of $\cW^c$, then the conclusion of Theorem \ref{thm_main} follows from Theorem \ref{thm_dichotomy}.
 
 Otherwise, by Theorem H of \cite{AVW2}, $\cW^c$ is absolutely continuous and $f=\psi^1$, where $\psi^t\colon M \to M$ is a smooth volume preserving Anosov flow. In particular, $\lbrace \psi^t \mid t\in \R \rbrace \subset \cZ(f)$.

 Now, if $g\in \cZ^c(f)$, then, by ergodicity of $f$, the map $g$ preserves $\vol$, and, hence, it preserves the disintegration of $\vol$ along $\cW^c$. Thus $g = \psi^t$ for some $t\in \R$.
 
 So $\lbrace \psi^t \mid t\in \R \rbrace = \cZ^c(f)$ and Theorem \ref{thm_main} follows from Lemma \ref{lem_quotient_is_finite}.
\end{proof}

\section{Proofs of Theorems \ref{thm_seifert} and \ref{thm_hyperbolic}} \label{sec_proof_of_B_C}

The two main results of \cite{BFFP} state that, if $f\colon M \to M$ is a partially hyperbolic diffeomorphism such that, either $f$ is homotopic to the identity and $M$ is Seifert, or that $f$ is dynamically coherent and $M$ is hyperbolic, then there exists $k\geq 1$ such that $f^k$ is a discretized Anosov flow.

Since $\cZ(f) \subset \cZ(f^k)$, we immediately deduce from Theorem \ref{thm_main} that, under the assumptions of Theorem \ref{thm_seifert} or Theorem \ref{thm_hyperbolic}, either $\cZ(f)$ is virtually $\lbrace f^n \mid n\in \Z\rbrace$ or $\cZ(f^k)$ is virtually $\R$ and $f^k$ embeds into an Anosov flow for some $k\geq 1$.

Thus, in order to finish proving Theorems \ref{thm_seifert} and \ref{thm_hyperbolic}, we only need to show that if $f^k$ is the time-$1$ map of an Anosov flow which is transitive on a Seifert or hyperbolic manifold, then the centralizer of $f$ is virtually $\R$.

This last step is given by the next lemma, which is in fact more general.

\begin{lemma}\label{lem_centralizer_of_f}
 Suppose that $f^k$ is the time-$1$ map of a transitive Anosov flow that is not a constant roof suspension of an Anosov diffeomorphism. Then $\cZ(f)$ is virtually $\R$.
\end{lemma}

In order to prove Lemma \ref{lem_centralizer_of_f}, we first need a result about topologically weak-mixing Anosov flows.

\begin{lemma}\label{lem_density}
 Let $\flot\colon M \to M$ be a topologically weak-mixing Anosov flow, then, for every $n>0$, the set of periodic orbits of $\flot$ that have period \emph{not} a multiple of $1/n$ is dense in $M$
\end{lemma}

\begin{proof}
 This is a simple consequence of the spatial equidistribution of orbits of periods between $T$ and $T +\eps$ for weak-mixing Anosov flow.

 We let $\cP$ be the set of periodic orbits of $\flot$. For any $\gamma \in \cP$, we let $\ell(\gamma)$ be the minimal period of $\gamma$. For any map $K\colon M \to \R$ that is continuous along the orbits of $\flot$, and any $\eps>0$, we have, by \cite[Proposition 7.3]{PP90},
 
\[
\frac{\sum_{T< \ell(\gamma) \leq T+\eps} \int_\gamma K}{\sum_{T< \ell(\gamma) \leq T+\eps} \ell(\gamma)} \rightarrow \int_M K d\mu_{BM}, \quad \text{as} \quad T\to +\infty,
\]
where $\mu_{BM}$ is the measure of maximal entropy of $\flot$.

Let $n>0$ be fixed and let $\cP_{\notin \frac{1}{n}\Z}$ be the set of periodic orbits of period not a multiple of $1/n$. If $\overline{\cP_{\notin \frac{1}{n}\Z}} \neq M$ then there would exists an open set $U$ that is missed by the orbits in $\cP_{\notin \frac{1}{n}\Z}$. Taking $K$ to be a smooth approximation of the characteristic function of $U$ and $\eps<1/n$, we would get that the left hand side of the above equation is zero along a subsequence, while the right hand side is strictly positive, as the measure of maximal entropy has full support. A contradiction.
\end{proof}

\begin{proof}[Proof of Lemma \ref{lem_centralizer_of_f}]
 Let $\flot\colon M \to M$ be the Anosov flow such that $f^k = \phi^1$, we will show that $f$ itself commutes with $\flot$ for any $t\in \R$ which will prove the claim (since $\cZ(f) \subset \cZ(f^k)$ and $\cZ(f^k)$ is virtually $\lbrace\flot \mid t\in \R\rbrace$).
 
 Since $f^k = \phi^1$, we have that, for any $m \in \Z$ and any $x\in M$, 
 \[
  f(\phi^m(x)) = \phi^m(f(x)).
 \]

 Now consider a periodic orbit $\gamma$ of $\flot$.

 If the period of $\gamma$ is irrational, then, by continuity, we have that for any $x\in\gamma$ and any $t\in \R$
\[
  f(\phi^{t}(x)) = \phi^{t}(f(x)).
 \]
On the other hand if the period of $\gamma$ is rational, say $p/n$, $\gcd(p,n)=1$, then for any $x\in \gamma$ and any $m\in \Z$, we have
 \[
  f(\phi^{m/n}(x)) = \phi^{m/n}(f(x)).
 \]

 Now, let $x\in M$, by Lemma \ref{lem_density} (which applies here because every Anosov flow which is not a suspension of an Anosov diffeomorphism by a constant roof function is topologically weak-mixing according to~\cite{Plante}), for every $n>1$, $x$ can be approximated by points $y^n_i\to x$, $i\to\infty$, on periodic orbits such that the periods of $y^n_i$ are either irrational or rational numbers $p_i/q_i$, $\gcd(p_i,q_i)=1$, with $q_i\to\infty$, $i\to\infty$. As we have seen above, at each $y^n_i$, the map $f$ commutes with at least $\phi^{m/q_i}$ for all $m\in\Z$.
 
 Passing to the limit as $n\to +\infty$, we obtain that $f$ commutes with every $\phi^t$ at $x$. Thus $\lbrace\flot \mid t\in \R\rbrace \subset \cZ(f)$, which proves the lemma.
\end{proof}

%

\providecommand{\bysame}{\leavevmode\hbox to3em{\hrulefill}\thinspace}
\providecommand{\MR}{\relax\ifhmode\unskip\space\fi MR }
\providecommand{\MRhref}[2]{%
  \href{http://www.ams.org/mathscinet-getitem?mr=#1}{#2}
}
\providecommand{\href}[2]{#2}

\end{document}